\documentclass[reqno,12pt]{amsart}
\usepackage{amssymb,amsmath}
\usepackage[dvips]{color}
\usepackage[dvips,final]{graphics}
\usepackage{epstopdf}
\usepackage{amscd}
\usepackage[margin=1.1in]{geometry}

\usepackage{graphicx}
\usepackage{xcolor}
\usepackage{tikz}

\renewcommand{\leq}{\leqslant}
\renewcommand{\le}{\leqslant}
\renewcommand{\geq}{\geqslant}
\renewcommand{\ge}{\geqslant}

\newcommand{\R}{\mathbb{R}}
\newtheorem{thm}{Theorem}
\newtheorem{lem}[thm]{Lemma}
\newtheorem{rmk}[thm]{Remark}
\newtheorem{cor}[thm]{Corollary}
\newtheorem{prop}[thm]{Proposition}
\newtheorem{definition}[thm]{Definition}
\begin{document}

\title{Neckpinch singularities \\
in fractional mean curvature flows}
\author[Eleonora Cinti]{Eleonora Cinti}
\address{E.C., Weierstra{\ss} Institute for Applied Analysis and Stochastics,
Mohrenstra{\ss}e 39,
10117 Berlin, Germany }

\email{cinti@wias-berlin.de}
\author[Carlo Sinestrari]{Carlo Sinestrari}

\address{C.S., Dipartimento di Ingegneria Civile e Ingegneria Informatica, 
Universit\`a di Roma Tor Vergata,
Via del Politecnico,
00133 Rome, Italy}  \email{sinestra@mat.uniroma2.it}
\author[Enrico Valdinoci]{Enrico Valdinoci}
\address{E.V., 
School of Mathematics and Statistics,
University of Melbourne,
813 Swanston St, Parkville VIC 3010, Australia,
Dipartimento di Matematica, Universit\`a di Milano, Via Cesare Saldini 50,
20133 Milan, Italy, and
Weierstra{\ss} Institute for Applied Analysis and Stochastics,
Mohrenstra{\ss}e 39,
10117 Berlin, Germany. }

\email{enrico.valdinoci@wias-berlin.de}
\thanks{
E.C. was supported by grants 
MTM2011-27739-C04-01 (Spain), 
2009SGR345 (Catalunya), and by the ERC Starting Grant 
``AnOptSetCon'' n. 258685.
E.C. and E.V. were supported by the ERC Starting Grant 
``EPSILON'' 
Elliptic Pde's and Symmetry of Interfaces and Layers for Odd Nonlinearities
n. 277749. C.S. was supported by the group 
GNAMPA of INdAM Istituto Nazionale di Alta Matematica.}
\begin{abstract}
In this paper we consider the evolution of boundaries of sets by a fractional mean curvature flow. We show that, for any dimension $n\geq 2$, there exist embedded hypersurfaces in $\R^n$ which develop a singularity without shrinking to a point. Such examples are well known for the classical mean curvature flow for $n \geq 3$. Interestingly, when $n=2$, our result provides instead a counterexample in the nonlocal framework to the well known Grayson's Theorem \cite{Gr87}, which states that any smooth embedded curve in the plane evolving by (classical) MCF shrinks to a point. The essential step in our construction is an estimate which ensures that a suitably small perturbation of a thin strip has positive fractional curvature at every boundary point.
\end{abstract}
\maketitle
\section{Introduction}

This paper is concerned with the study of a \textit{nonlocal mean curvature flow}. More precisely, we want to study the evolution $E_t$, for time $t>0$ of an initial set $E_0$, such that the velocity of a point $x\in \partial E_t$ in the outer normal direction $\nu$ is given by the quantity $-H^s_E$, where $H^s_E$ denotes the fractional mean curvature of $E$, that is, for any $x\in \partial {E_t}$ we have
\begin{equation}\label{NMCF}
\partial_ t x\cdot \nu=-H^s_{E_t}.
\end{equation}

For a real parameter $s\in (0,1)$, we recall that the fractional mean curvature of a set $E$ at a point $x\in \partial E$ is defined as follows
\begin{equation}\label{H_s}
H^s_E(x):=\int_{\R^n}\frac{\chi_{\mathcal C  E}(y)-\chi_{E}(y)}{|x-y|^{n+s}}\,dy,
\end{equation}
where $\chi_A$ denotes the characteristic function of the set $A$, $\mathcal C A$ denotes the complement of $A$, and the integral above has to be understood in the principal value sense.

This evolution is the natural analogue in the nonlocal setting of the classical mean curvature flow, which has been widely studied in the last decades, see e.g. \cite{Eck,Man}.  While the classical mean curvature flow is the $L^2$-gradient flow of the usual perimeter functional, it can be proved that \eqref{NMCF} is the $L^2$-gradient flow of the \textit{fractional perimeter}, which was first introduced in \cite{CRS} on the basis of motivations coming from interfaces in physical models and probabilistic processes. In the same paper, suitable density estimates, a monotonicity formula, and some regularity results for minimizers were established. The question of the regularity for minimizers of the fractional perimeter was also addressed in several recent works, see~\cite{BFV,CSV,CV,SV}. Further motivations for the study of \eqref{NMCF} come from dislocation dynamics and phase-field theory for fractional reaction-diffusion equations, see \cite{Im}.

The classical mean curvature flow is a quasilinear parabolic problem and a local existence result holds for smooth solutions starting from any compact regular initial surface, see \cite{GH,Man}. As time evolves, solutions typically develop singularities  due to curvature blowup. Several notions of generalized solutions have been introduced to study the flow after the onset of singularities. Particularly relevant for our purposes are the definitions by Chen, Giga, Goto \cite{CGG} and by Evans and Spruck \cite{ES}, based on the level set approach and the notion of viscosity solutions.

For the fractional mean curvature flow, local smooth solutions are also expected to exist, but no proof of this property is available yet. There are existence results for viscosity solutions, first obtained by Imbert \cite{Im}, and later extended to more general nonlocal flows by Chambolle, Morini, Ponsiglione in \cite{CMP,CMP2} and by Chambolle, Novaga, Ruffini \cite{CNR}.

The analysis of the formation of singularities is an important topic in classical mean curvature flow. A pioneering result in this framework was obtained by Huisken \cite{Hui}, who showed that a closed convex surface in $\R^n$, with $n >2$, remains convex along the evolution and shrinks to a point in finite time. If convexity is dropped, then other kinds of singular behaviour may occur. A standard example is the so called \textit{neckpinch}. The idea is to consider a surface which looks like two large balls connected by a very thin cylindrical neck, so that, in dimension $n>2$, the mean curvature in the neck is much larger than the one in the balls, hence the radius of the neck goes to zero faster than the radius of the balls. The existence of this type of surfaces was first proved by Grayson \cite{Gr} and later considered with a simplified proof by Ecker \cite{Eck}, see also Angenent \cite{An}. A similar construction for a flow driven by a different curvature function was done in \cite{AS}.

When $n=2$, a result analogous to \cite{Hui} for convex curves was proved by Gage and Hamilton \cite{GH}. However, a stronger result holds in this dimension in the classical case. In fact, Grayson \cite{Gr87} showed that {\em any smooth closed embedded curve in the plane becomes convex in finite time under the flow}, and therefore, by \cite{GH}, {\em shrinks smoothly to a point}. Thus, all other kinds of singularities are ruled out for embedded curves.

In this paper we construct examples of neckpinch singularities for the fractional mean curvature flow. More precisely, we obtain the 
following result:

\begin{thm}\label{thm-neckpinch}
Let $n\geq 2$. There exists an embedded hypersurface $\mathcal M_0$ in $\R^n$ such that the viscosity solution of the fractional mean curvature flow \eqref{NMCF} starting from $\mathcal M_0$ does not shrink to a point.
\end{thm}

We point out that,
in contrast to the classical case, our construction can be made in any dimension, in particular for $n=2$, showing that Grayson's theorem fails in the nonlocal case. It also shows that the distance comparison property for curves for the classical flow proved by Huisken \cite{Hui3} no longer holds in the nonlocal
case. Our result has some analogies with the one of \cite{CFSW}, where Delaunay-type periodic curves in the plane with constant fractional mean curvature are constructed, while in the classical case such objects exist only in dimension $n \geq 3$.

Our results deals with the viscosity solution of \eqref{NMCF} because we lack a local existence result for smooth solutions. However, Theorem \ref{thm-neckpinch} also implies that, if the hypersurface we construct has a local smooth evolution, then it \textit{develops singularities} before shrinking to a point.

A crucial step in the proof of Theorem \ref{thm-neckpinch} is provided by the following result, of independent interest: if a set $E$ is contained in a strip and its boundary $\partial E$ has sufficiently small slope and small classical curvatures, then the fractional mean curvature of $E$ is bounded below by a positive universal constant, which depends only on $s$ and on the dimension $n$. This is the content of Proposition \ref{STIMA:CUR:FLAT}. In particular, a thin set can have all negative classical principal curvatures at some point, but positive fractional mean curvature.

Other recent contributions in the study of the fractional (and more general nonlocal) mean curvature flows are the articles~\cite{CSou, CNR, SaezV}.
 In \cite{CSou}, the convergence of a class of threshold dynamics approximations to moving fronts was established. In particular, threshold dynamics associated to fractional powers of the Laplacian of order $s\in (0,2)$ were considered: interestingly, when $s\in [1,2)$ the resulting interface moves by a (weighted) mean curvature flows, while when $s\in (0,1)$ it moves by a fractional mean curvature flow.
In \cite{CNR}, the results contained in \cite{CSou} have been extended to the anisotropic case, and it was proved that convexity is preserved as in the classical case. Finally, in \cite{SaezV} smooth solutions to the fractional mean curvature flow were studied and the evolution equations for several geometric local and nonlocal quantities were computed. The particular cases of entire graphs and star-shaped surfaces were considered, obtaining striking analogies with the properties of the classical case.
%Finally, in \cite{SaezV} smooth solutions to the fractional mean curvature flow were studied. In particular, a comparison principle, uniqueness of smooth solutions, and extinntion in finite time for compact solutions were established. Moreover the evolution equations for several geometric local and nonlocal quantities were computed, and the particular cases of entire graphs and star-shaped surfaces were considered.

The paper is organized as follows:

\begin{itemize}
\item In Section 2 we describe the level set approach in the study of nonlocal mean curvature flows, we recall the notion of viscosity solutions and the statement of the comparison principle;
\item In Section 3 we establish the estimates on the nonlocal mean curvature of perturbed strips, which will be used in the proof of our main result;
\item In Section 4, we prove our main result Theorem \ref{thm-neckpinch}.
\end{itemize}

\section{Viscosity solutions and comparison principles via the level set approach}

In this Section we recall the notion of viscosity solutions for the fractional mean curvature flow, which is based on the
%was first introduced in \cite{Im} and later extended to a wide class of local and nonlocal geometric flows in \cite{CMP}. In both works the classical 
\textit{level sets approach}. 
%This technique was first considered by Evans and Spruk in \cite{ES} in order to define a \textit{weak} notion of solutions to the classical mean curvature flow, which let one study solutions of the flow also after the formation of singularities.
The idea is the following: given an initial surface $\mathcal M_0=\partial E_0$, we choose any continuous function $u_0:\R^n\rightarrow \R^n$ such that 
\begin{equation}\label{level-set}\mathcal M_0=\{x\in \R^n\,:\,u_0(x)=0\}.\end{equation}
The geometric equation satisfied by the evolution $\mathcal M_t$ of $\mathcal M_0$ can then be translated into an equation satisfied by a function $u(x,t)$, where $u(x,0)=u_0(x)$ and at each time
\begin{equation}\label{M_t}\mathcal M_t=\{x\in \R^n\,:\,u(\cdot, t)=0\}.\end{equation}
More precisely, the \textit{level set equation} satisfied by $u$ is 
\begin{equation}\label{level-set-eq}
\partial_t u+H^s[x,u(\cdot,t)]|Du(x,t)|=0\quad \mbox{in}\;\;\R^n\times (0,+\infty),
\end{equation}
where $u$ satisfies the initial condition
$$u(x,t)=u_0(x)\quad \mbox{in}\;\;\R^n.$$
Here and in the following we denote by $H_s[x,u(\cdot,t)]$ the fractional mean curvature of the superlevel set of $u(\cdot, t)$ at the point $x$, i.e.
$$H^s[x,u(\cdot,t)]=H^s_{\{y\in \R^n\,:\,u(y,t)>u(x,t)\}}(x).$$
Of course the definition of $\mathcal M_t$ is well posed if one shows that equation \eqref{level-set-eq} has a unique solution, and definition \eqref{level-set} does not depend on the initial choice of the function $u_0$. These basic properties were established in \cite{ES} for the classical mean curvature flow and in \cite{Im} for the fractional one.

Since the nonlocal case is not so standard and for the sake of completeness, we recall here below all the rigorous definitions and the basic results in \cite{Im} (see also \cite{CMP}). 
Let $\mathcal M=\{x\in \R^n\,:\,u(x)=0\}=\partial\{x\in \R^n\,:\,u(x)>0\}$. If $u\in C^{1,1}$ and $Du\neq 0$, we can define the following quantities
\begin{equation}\label{k*}
\begin{split}
k^*[x,\mathcal M]&=k^*[x,u]=\int_{\R^n}\frac{\chi_{\{u(x+z)\geq u(x)\}}(z)  -\chi_{\{u(x+z)<u(x)\}}(z)}{|z|^{n+s}}\,dz,
\\
k_*[x,\mathcal M]&=k_*[x,u]= \int_{\R^n}\frac{\chi_{\{u(x+z)> u(x)\}}(z)  -\chi_{\{u(x+z) \leq u(x)\}}(z)}{|z|^{n+s}}\,dz.
\end{split}
\end{equation}
It is easy to see that if $u\in C^{1,1}$ and its gradient $Du$ does not vanish on $\{z\in \R^n\,:\,u(z)=u(x)\}$, then $k^*$ are finite and
$$k^*[x,u]=k_*[x,u]=-H_s[x,u].$$

We can now give the definition of \textit{viscosity solution} for \eqref{level-set-eq} (see \cite{Im}, Sec. 3).

\begin{definition}
\begin{itemize}
\item[i)] An upper semicontinuous function $u:[0,T]\times \R^n$ is a \textit{viscosity subsolution} of \eqref{level-set-eq} if for every smooth test function $\phi$ such that $u-\phi$ admits a global zero maximum at $(t,x)$, we have
\begin{equation}\label{sub}
\partial_t \phi\leq k^*[x,\phi(\cdot,t)]|D\phi|(x,t)
\end{equation}
if $D\phi(x,t)\neq 0$, and $\partial_t\phi(x,t)\leq 0$ if not.
\item [ii)] A lower semicontinuous function $u:[0,T]\times \R^n$ is a \textit{viscosity supersolution} of \eqref{level-set-eq} if for every smooth test function $\phi$ such that $u-\phi$ admits a global zero minimum at $(t,x)$, we have
\begin{equation}\label{super}
\partial_t \phi\geq k_*[x,\phi(\cdot,t)]|D\phi|(x,t)
\end{equation}
if $D\phi(x,t)\neq 0$, and $\partial_t\phi(x,t)\geq 0$ if not.
\item[iii)] A locally bounded function $u$ is a \textit{viscosity solution} of \eqref{level-set-eq} if its upper semicontinuous envelope is a subsolution and its lower semicontinuous envelope is a supersolution of \eqref{level-set-eq}.
\end{itemize}
\end{definition}

\begin{rmk}{\rm
It is easy to verify that any classical subsolution (respectively supersolution) is in particular a viscosity subsolution (respectively supersolution). 
}\end{rmk}

We can now state the comparison principles, that we will use later on in Section \ref{neckpinch}. 

\begin{prop}[Theorem 2 in \cite{Im}]\label{comparison}
Suppose that the initial datum $u_0$ is a bounded and Lipschitz continuous function. Let $u$ (respectively $v$) be a bounded viscosity subsolution (respectively supersolution) of \eqref{level-set-eq}.

If $u(x,0)\leq u_0(x)\leq v(x,0)$, then $u\leq v$ on $\R^n\times (0,+\infty)$.
\end{prop}

In Theorem 3 of \cite{Im}  existence and uniqueness of viscosity solutions of \eqref{level-set-eq} were proven. The prove of existence uses
Perron's method, while uniqueness relies on the comparison principle stated in Proposition \ref{comparison}.
Finally in \cite{Im}, Theorem 6, the consistency of Definition \eqref{M_t} is established, showing that if $u$ and $v$ are two viscosity solutions of \eqref{level-set-eq} with two different initial data $u_0$ and $v_0$ which have the same zero-level set, then for every time $t>0$ also $u(\cdot, t)$ and $v(\cdot, t)$ have the same zero level set.

The uniqueness and the consistency results allow to define the fractional mean curvature flow for $\mathcal M_t$ by using the solution $u(\cdot,t)$ of \eqref{level-set-eq}.

\section{The nonlocal curvature of perturbed strips}

The goal of this section is to give estimates
on the nonlocal curvature of sets with small slopes
and small classical curvatures which are contained in a strip.
The result is rather general, but we focus on
a particular case for the sake of concreteness:

\begin{prop}\label{STIMA:CUR:FLAT}
Let~$\kappa$, $\eta>0$.
Let~$E_-$, $E_+\subset\R^n$. Assume that~$E_+\cap E_-=\varnothing$, that
\begin{equation}\label{CONTE}
%\begin{split}
E_- \supseteq \{ x_n \le -1 \} \quad {\mbox{and }}\quad
E_+ \supseteq \{ x_n \ge 1 \}.
%\end{split}
\end{equation}
Suppose also that the boundaries of~$E_-$ and~$E_+$
are of class~$C^2$, with classical directional curvatures
bounded by~$\kappa$.

In addition, let~$\nu_-$ and~$\nu_+$ be the exterior
normals of~$E_-$ and~$E_+$ respectively, and
assume that
\begin{equation}\label{CONTE:0}
|\nu_-\cdot e_i|\le \eta \quad{\mbox{ and }}\quad
|\nu_+\cdot e_i|\le \eta \quad{\mbox{ for every }} i\in\{1,\dots,n-1\}.
\end{equation}
Let~$E:= \R^n\setminus (E_-\cup E_+)$. Then, 
there exist~$c_0$, $\kappa_0$, $\eta_0>0$, depending on~$n$ and~$s$,
such that for any~$x\in\partial E$
$$ H^s_E (x) \ge c_0,$$
provided that~$\kappa\in [0,\,\kappa_0]$ and~$\eta\in[0,\,\eta_0]$.
\end{prop}

Before giving the proof of Proposition \ref{STIMA:CUR:FLAT}, we state the following Lemma that will be useful later.
Basically, it states that a set, whose boundary has small enough slopes and classical curvatures bounded by $\kappa$, can be trapped between paraboloids of opening a multiple of $\kappa$.  

\begin{lem}\label{SP:SACL}
Let~$\kappa$, $\eta>0$ and~$G\subset\R^n$, with
boundary of class~$C^2$, with classical directional curvatures
bounded by~$\kappa$ and exterior
normal~$\nu$ such that~$|\nu\cdot e_i|\le \eta$ for any~$i\in\{1,\dots,n-1\}$.

Suppose also that~$0\in\partial G$.
Then, fixed~$\rho>0$, there exist~$\eta_0>0$,
depending on~$n$, and~$\kappa_0(\rho)$, depending on~$n$ and~$\rho$,
such that, if~$\kappa\in[0,\,\kappa_0(\rho)]$ and~$
\eta\in[0,\,\eta_0]$, then
$$ {\mathcal{R}}({{G}}\cap B_\rho\big)\subseteq
\{ x_n\le C\kappa\,|x'|^2\},$$
for a suitable rotation~${\mathcal{R}}$, which differs from the identity
less than~$C\eta$, for a suitable~$C>0$, depending on~$n$.
\end{lem}

\begin{proof} We notice that the result is true for some~$\rho_0$
universal, due to the Implicit Function Theorem,
and so for all~$\rho\in(0,\rho_0]$.

Let now~$\rho>\rho_0$. Let
$$ G_\rho := \frac{\rho_0}{\rho}\,G=
\{ y\in\R^n {\mbox{ s.t. }} \rho \,\rho_0^{-1}\,y\in G\}.$$
Notice that the slope of the normal of~$G_\rho$ is the same as the one of~$G$,
and so it is bounded by~$\eta$. On the other hand,
the curvatures of~$G_\rho$ are~$\rho\,\rho_0^{-1}$ times the curvatures of~$G$,
and therefore are bounded by~$\rho\,\rho_0^{-1}\,\kappa$.

Accordingly, if~$\eta\in [0,\,\eta_0]$ and~$\rho\,\rho_0^{-1}\,\kappa\in
[0,\,\kappa_0(\rho)]$ (i.e. if~$\kappa\in[0,\,\rho_0\,\rho^{-1}\,\kappa_0(\rho_0)]$),
we have that the claim holds true for~$G_\rho$ in~$B_{\rho_0}$, namely
$$ {\mathcal{R}}({{G}}_\rho\cap B_{\rho_0}\big)\subseteq
\{ y_n\le C\,\rho\,\rho_0^{-1}\,\kappa\,|y'|^2\},$$
where ${\mathcal{R}}$ is a rotation as in the statement of
Lemma~\ref{SP:SACL}.

Using the change of scale~$x:=\rho\,\rho_0^{-1}\,y$,
this gives that
$$ {\mathcal{R}}({{G}}\cap B_{\rho}\big)=\frac{\rho}{\rho_0}\,
{\mathcal{R}}({{G}}_\rho\cap B_{\rho_0}\big)
\subseteq\frac{\rho}{\rho_0}\,
\{ y_n\le C\,\rho\,\rho_0^{-1}\,\kappa\,|y'|^2\}=\{ x_n\le C\kappa\,|x'|^2\},
$$
as desired.
\end{proof}

\begin{proof}[Proof of Proposition \ref{STIMA:CUR:FLAT}] We fix~$x\in\partial E$. Without loss of
generality, we may suppose that~$x\in\partial E_+$.
We use the notation~$\nu:=\nu_+(x)$.
Also, we set~$\chi:=\chi_E-\chi_{\mathcal C E}$. 
For short, we also denote by~$C$ a positive constant,
depending on~$n$ and~$s$, which we take the liberty of modifying
from line to line.

We consider
the two slabs
\begin{eqnarray*}
&& S := \{ p\in\R^n {\mbox{ s.t. }} |(p-x)\cdot\nu|\le 4\}
\\{\mbox{and }} && S_\star
:= \{ p\in\R^n {\mbox{ s.t. }} |p\cdot e_n|\le 1\}.
\end{eqnarray*}
Notice that, by~\eqref{CONTE},
\begin{equation}\label{CONTE:2}
E\subseteq S_\star.
\end{equation}
We also observe that
\begin{equation}\label{CONTE:3}
S_\star\setminus S\subseteq \R^n\setminus B_{C/\eta}(x).
\end{equation}
Indeed, if~$p\in S_\star\setminus S$,
we have that~$|(p-x)\cdot\nu|> 4$
and~$|p\cdot e_n|\le 1$. Thus, we recall~\eqref{CONTE:0}
and~\eqref{CONTE:2},
we write~$\nu=(\nu\cdot e_1)e_1+\dots+(\nu\cdot e_n)e_n$
and we find that
\begin{eqnarray*}
4 &<& \left| (p-x)\cdot \sum_{i=1}^n (\nu\cdot e_i)e_i\right|\\
&\le& \eta\,\sum_{i=1}^{n-1} |(p-x)\cdot e_i|
+|(p-x)\cdot e_n| \\
&\le& \eta\,(n-1)\,|p-x| + |p\cdot e_n|+|x\cdot e_n| \\
&\le& \eta\,(n-1)\,|p-x| + 1+1,
\end{eqnarray*}
which implies~\eqref{CONTE:3} as long as~$\eta$ is sufficiently small.

Then, in view of~\eqref{CONTE:2}
and~\eqref{CONTE:3}, we have that
$$ E\setminus S\subseteq
S_\star\setminus S\subseteq \R^n\setminus B_{C/\eta}(x)$$
and therefore
\begin{equation}\label{CIxS:00:1}
\int_{E\setminus S} \frac{dy}{|x-y|^{n+s}}\le
\int_{ \R^n\setminus B_{C/\eta}(x) } \frac{dy}{|x-y|^{n+s}}
\le C\eta^s.
\end{equation}
Now we set
$$ G:= \{ p\in B_{10}(x) {\mbox{ s.t. }} (p-x)\cdot \nu > 4 \}.$$
Notice that~$G$ is a circular sector, so, in particular, we
have that
\begin{equation}\label{uj4}
|G|\ge C.\end{equation}
Also, by construction,
\begin{equation}\label{uj5}
G\cap S=\varnothing.
\end{equation}
We claim that
\begin{equation}\label{uj6}
G\subseteq \R^n\setminus E.
\end{equation}
To check this, let~$p\in G$. Then, by~\eqref{CONTE:0} and~\eqref{CONTE:2},
\begin{eqnarray*}
4 &<& \sum_{i=1}^n (\nu\cdot e_i)(p-x)\cdot e_i\\
&\le& \eta\,(n-1)\,|p-x|+(p-x)\cdot e_n \\
&\le& 10\eta\,(n-1) + p\cdot e_n +1,
\end{eqnarray*}
and so, if~$\eta$ is sufficiently small, we deduce that~$p\cdot e_n
\ge 2$. Hence, we have that~$p\in E_+$, which in turn implies~\eqref{uj6}.

Now, from~\eqref{uj5}
and~\eqref{uj6}, we get that
$$ (\R^n\setminus E)\setminus S\;\supseteq\;
G\setminus S\;=\; G.$$
Notice also that if~$y\in G$ then~$|x-y|\le 10$. As a consequence
of these observations and~\eqref{uj4},
we infer that
\begin{equation*}
\int_{(\R^n\setminus E)\setminus S} \frac{dy}{|x-y|^{n+s}}\ge 
\int_{G} \frac{dy}{|x-y|^{n+s}}\ge \frac{|G|}{10^{n+s}}\ge C.
\end{equation*}
Combining this with~\eqref{CIxS:00:1}, we conclude that
\begin{equation}\label{CIxS:00:2}
\int_{\R^n\setminus S} \frac{\chi(y)\,dy}{|x-y|^{n+s}}=\int_{\R^n\setminus S} \frac{\chi_E(y)-\chi_{\mathcal C E}(y)}{|x-y|^{n+s}}\,dy\le C\eta^s - C_\star.
\end{equation}
Here we denote by~$C_\star>0$ a ``special'' constant
(to distinguish it from the ``other'' constants
just denoted by~$C$): indeed, $C_\star$ will produce
the desired result after small perturbations, as we will see in the sequel.

Now we analyze the contributions inside~$S$. For this purpose,
we take an additional parameter
\begin{equation}\label{KA:INT}
\rho\in \left(0,\; \,\min\left\{|\log\kappa|,\,\frac1\eta\right\}\, \right).
\end{equation}
Such~$\rho$ will be taken appropriately large (in dependence of~$C_\star$):
then~$\kappa$ and~$\eta$
have to be chosen sufficiently small, in such a way
that~$\rho$ belongs to the interval stated in~\eqref{KA:INT}. That is, the parameter~$\rho$
is taken large with respect to the constants, then~$\kappa$ and~$\eta$
will be taken to be small in dependence on the constants
and on~$\rho$.

Thanks to the curvature assumption on~$\partial E_+$,
we know that the boundary of~$E_+$ inside~$B_\rho(x)$
is trapped between paraboloids of opening~$C\kappa$:
more precisely, by Lemma \ref{SP:SACL} there exists a rigid motion~${\mathcal{R}}$
with~${\mathcal{R}}(x)=0$ and such that
\begin{equation}\label{9ugAHHHfA}
\begin{split}&
{\mathcal{R}} \big( E_+\cap B_\rho(x)\big) \supseteq
\left\{ p=(p',p_n)\in B_\rho {\mbox{ s.t. }} p_n \ge C\kappa\,|p'|^2
\right\}=:{\mathcal{P}}_1,\\ &
{\mathcal{R}} \big( (\partial E_+)\cap B_\rho(x)\big) \subseteq
\left\{ p=(p',p_n)\in B_\rho {\mbox{ s.t. }} |p_n| \le {C\kappa\,|p'|^2}
\right\}=:{\mathcal{P}}_2\\
{\mbox{and }}\quad& {\mathcal{R}} \big( E_+\cap B_\rho(x)\big)\subseteq
{\mathcal{P}}_1\cup{\mathcal{P}}_2.
\end{split}\end{equation}
Notice that~${\mathcal{R}}$ is the composition of a translation
that sends~$x$ to~$0$ and then a rotation which sets~$\nu$
in the vertical direction. Therefore, we can write~${\mathcal{R}}(y):=
{\mathcal{R}}_\star \,(y-x)$, where~${\mathcal{R}}_\star$ is a rotation
which differs from the identity less than~$C\eta$, due to~\eqref{CONTE:0}.

\begin{figure}[h]
	\centering

	%%ZEICHNUNG1
	%%define overall image width here
	\resizebox{10cm}{!}{
	
		\begin{tikzpicture}[scale=1]
		
			%%include pdf picture
			\node (myfirstpic) at (0,0) {\includegraphics[scale=0.6]{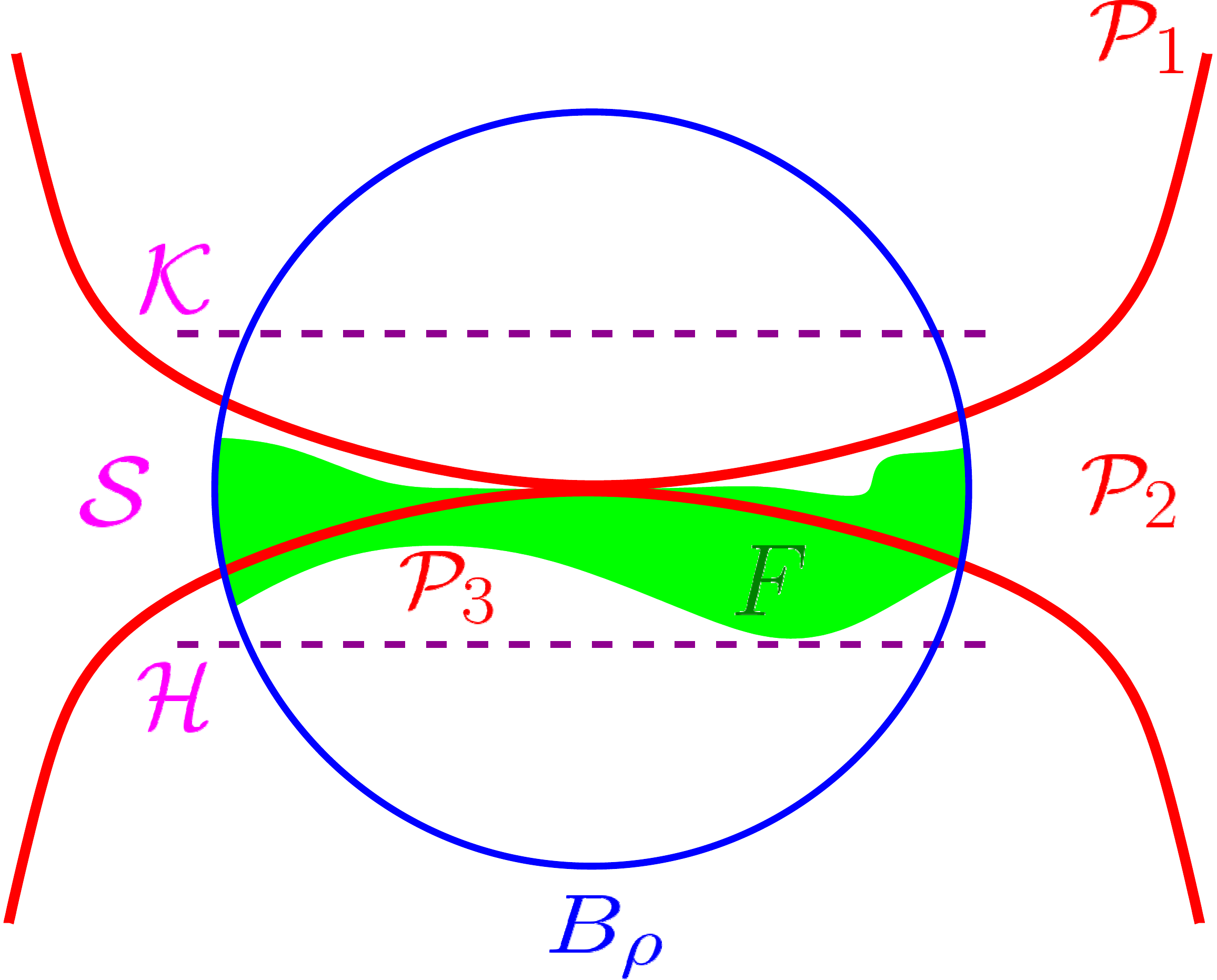}};
			
			%%put latex math over picture

		\end{tikzpicture}
		
	}
	\caption{}
\end{figure}

We set
\begin{eqnarray*}&& F:={\mathcal{R}} \big( E\cap B_\rho(x)\big),\\
&& {\mathcal{H}}:=
\left\{ p=(p',p_n)\in B_\rho {\mbox{ s.t. }} p_n <-4
\right\},
\\ && {\mathcal{K}}:=
\left\{ p=(p',p_n)\in B_\rho {\mbox{ s.t. }} p_n >4
\right\}
\\{\mbox{and }}&& {\mathcal{P}}_3:=B_\rho\setminus
\big({\mathcal{P}}_1\cup{\mathcal{P}}_2\cup{\mathcal{H}}\big).
\end{eqnarray*}

By construction, we see that
\begin{equation}\label{SUR:SP}
{\mathcal{R}}(S\cap B_\rho(x))=B_\rho\setminus({\mathcal{H}}\cup{\mathcal{K}})=
\left\{ p=(p',p_n)\in B_\rho {\mbox{ s.t. }} |p_n| \le4
\right\}=:{\mathcal{S}}.
\end{equation}
We also observe that
\begin{equation} \label{FUINVL}
{\mathcal{P}}_3\cap {\mathcal{S}} =
-({\mathcal{P}}_1\cap{\mathcal{S}}),\end{equation}
up to sets of null measure.
To check this, let~$p\in {\mathcal{P}}_3\cap {\mathcal{S}}$
and~$q:=-p$. Then we have that~$|q_n|=|p_n|\le4$, thus~$q\in{\mathcal{S}}$.
Furthermore, since~$p\not\in {\mathcal{P}}_1\cup {\mathcal{P}}_2$,
we have that~$p_n<-C\kappa\,|p'|^2$. Accordingly,
$$ q_n =-p_n>C\kappa\,|p'|^2 =C\kappa\,|q'|^2,$$
and so~$ q\in {\mathcal{P}}_1$.
This shows that~$q\in{\mathcal{P}}_1\cap{\mathcal{S}}$, and so
\begin{equation}\label{IOKHAYjaA}
{\mathcal{P}}_3\cap {\mathcal{S}} \subseteq
-({\mathcal{P}}_1\cap{\mathcal{S}}).
\end{equation}
Now, let us take~$q\in ({\mathcal{P}}_1\cap{\mathcal{S}})$
and set~$p:=-q$. We have that~$|p_n|=|q_n|\le4$,
hence~$p\in {\mathcal{S}}$. Moreover,
$$ p_n =-q_n\le -C\kappa\,|q'|^2 =-C\kappa\,|p'|^2.$$
Since we can remove the case in which~$p_n=-C\kappa\,|p'|^2$
(being this a set of null measure), we can conclude that~$p_n<
-C\kappa\,|p'|^2$ and so~$p\not\in{\mathcal{P}}_1\cup{\mathcal{P}}_2$.
This says that
$$ p\in {\mathcal{S}}\setminus\big( {\mathcal{P}}_1\cup{\mathcal{P}}_2\big)
\subseteq \big( B_\rho\setminus{\mathcal{H}}\big)\setminus
\big( {\mathcal{P}}_1\cup{\mathcal{P}}_2\big)
\subseteq B_\rho\setminus
\big( {\mathcal{P}}_1\cup{\mathcal{P}}_2\cup{\mathcal{H}}\big)
={\mathcal{P}}_3.$$
Hence~$p\in{\mathcal{P}}_3\cap{\mathcal{S}}$, up to negligible sets.
Thus, we have shown that
$$ -({\mathcal{P}}_1\cap{\mathcal{S}})\subseteq
{\mathcal{P}}_3\cap {\mathcal{S}} ,$$
up to negligible sets. Combining this with~\eqref{IOKHAYjaA},
we conclude the proof of~\eqref{FUINVL}.

Also, we notice that
\begin{equation}\label{SUR}
{\mathcal{H}}\cap \big({\mathcal{P}}_1\cup{\mathcal{P}}_2\big)=\varnothing
\quad{\mbox{ and }}\quad
{\mathcal{K}}\subseteq {\mathcal{P}}_1.
\end{equation}
Indeed, if~$p\in
\big({\mathcal{P}}_1\cup {\mathcal{P}}_2\big)$, then
$$ p_n \ge -{C\kappa\,|p'|^2} \ge -C\kappa\rho^2\ge -1$$
if~$\kappa$ sufficiently small (possibly in dependence of~$\rho$),
hence~$p\not\in{\mathcal{H}}$.
Similarly, if~$p\in {\mathcal{K}}$ then
$$ p_n -C\kappa\,|p'|^2\ge 4-C\kappa\rho^2>0$$
and so~$p\in{\mathcal{P}}_1$. These observations prove~\eqref{SUR}.

We also point out that
\begin{equation} \label{IJ5678-UAYT}
F =
{\mathcal{R}} \big( E\cap B_\rho(x)\big)
\subseteq
{\mathcal{R}} \big( (\R^n\setminus E_+)\cap B_\rho(x)\big) 
= \R^n\setminus\Big(
{\mathcal{R}} \big( E_+\cap B_\rho(x)\big) \Big)
\subseteq\R^n\setminus {\mathcal{P}}_1,\end{equation}
thanks to~\eqref{9ugAHHHfA}, and therefore
\begin{equation}\label{ikAIII:A}
B_\rho\setminus F\supseteq {\mathcal{P}}_1.\end{equation}
In addition, we claim that
\begin{equation}\label{IND:FF}
F\subseteq {\mathcal{P}}_2\cup {\mathcal{P}}_3.
\end{equation}
Indeed, if~$p\in F$, then~$p={\mathcal{R}}(y)=
{\mathcal{R}}_\star \,(y-x)$ for some~$y\in E$. Therefore,
if~$I_n\in{\rm Mat}(n\times n)$ denotes the identity matrix,
recalling~\eqref{CONTE:2}
we have that
\begin{eqnarray*} 
p_n &=& \big({\mathcal{R}}_\star\,(y-x)\big)\cdot e_n
\\ &\ge& (y-x)\cdot e_n -
\Big| \big(({\mathcal{R}}_\star-I_n)\,(y-x)\big)\cdot e_n\Big| \\
&\ge& -2 -\eta\rho
\\ &>& -4,
\end{eqnarray*}
provided that~$\eta$ is sufficiently small (possibly in dependence of~$\rho$).
This shows that~$p\not\in{\mathcal{H}}$.
As a consequence, we see that
\begin{equation*}
p\in{\mathcal{P}}_1\cup{\mathcal{P}}_2\cup{\mathcal{P}}_3.\end{equation*}
This and~\eqref{IJ5678-UAYT}
establish the validity of~\eqref{IND:FF}.

Now, we use the change of variable~$z={\mathcal{R}}(y)$ and formulas
\eqref{SUR:SP}, \eqref{ikAIII:A} and~\eqref{IND:FF}
to see that
\begin{equation}\label{2IaAJLLAK}
\begin{split}
\int_{S\cap B_\rho(x)} \frac{\chi(y)\,dy}{|x-y|^{n+s}}
\,&=\int_{ {\mathcal{S}} } \frac{\chi_F(z)-\chi_{B_\rho\setminus F}(z)\,dz}{
|z|^{n+s}} \\
&\le 
\int_{ {\mathcal{P}}_2 } \frac{dz}{|z|^{n+s}}
+\int_{ {\mathcal{P}}_3\cap{\mathcal{S}} } \frac{dz}{|z|^{n+s}}
-\int_{ {\mathcal{P}}_1\cap{\mathcal{S}} } \frac{dz}{|z|^{n+s}}.
\end{split}
\end{equation}
Moreover, from~\eqref{FUINVL} and a reflection in the vertical variable, we
have that
$$ \int_{ {\mathcal{P}}_3\cap{\mathcal{S}} } \frac{dz}{|z|^{n+s}}=
\int_{ {\mathcal{P}}_1\cap{\mathcal{S}} } \frac{dz}{|z|^{n+s}}. $$
Accordingly, \eqref{2IaAJLLAK} becomes
\begin{equation}\begin{split}
\label{98uhAaasdGH}
& \int_{S\cap B_\rho(x)} \frac{\chi(y)\,dy}{|x-y|^{n+s}}
\le\int_{ {\mathcal{P}}_2 } \frac{dz}{|z|^{n+s}}
\\ &\qquad\le \int_{ |z'|\le \rho }\,dz'\;
\int_{|z_n|\le C\kappa\,|z'|^2} \,dz_n\; |z'|^{-n-s}
=C\kappa\,\int_{ |z'|\le \rho }\,dz' \;|z'|^{2-n-s}
\\ &\qquad=C\kappa \rho^{2-s}.
\end{split}\end{equation}
On the other hand, we have that
\begin{equation*}
\int_{S\setminus B_\rho(x)} \frac{\chi(y)\,dy}{|x-y|^{n+s}}\le
\int_{\R^n\setminus B_\rho(x)} \frac{dy}{|x-y|^{n+s}}
= \frac{C}{\rho^s} .\end{equation*}
This and~\eqref{98uhAaasdGH} yield
$$ \int_{S} \frac{\chi(y)\,dy}{|x-y|^{n+s}}\le C\,\left(
\kappa \rho^{2-s}+\frac{1}{\rho^s}
\right).$$
Combining this with~\eqref{CIxS:00:2}, we see that
$$ -H^s_E (x)=\int_{\R^n\setminus S} \frac{\chi(y)\,dy}{|x-y|^{n+s}}\le
C\,\left(\eta^s +\kappa \rho^{2-s}+\frac{1}{\rho^s}
\right) - C_\star \le 
C\,\left(\eta^s +\kappa \rho^{2-s}
\right) - \frac{C_\star}{2},$$
provided that~$\rho$ is sufficiently large.

Hence, if~$\eta$ and~$\kappa$ are small enough, 
we obtain that~$-H^s_E (x)\le -C_\star/4$,
as desired.
\end{proof}

\begin{cor}\label{arctan}
Let~$\epsilon,\,\delta>0$ and
$$ E:= \left\{ x=(x',x_n)\in\R^n {\mbox{ s.t. }}
|x_n|< \epsilon + \frac{2}{\pi}\,\arctan \big(\delta \,|x'|^2\big)
\right\} . $$
Then, if~$\epsilon$ and $\delta$ are sufficiently small, we have that
$$ \inf_{x\in \partial E} H^s_E(x)\geq c_0 > 0, $$
for some $c_0$ depending only on $s$ and $n$.

\end{cor}

\begin{proof} We define
\begin{eqnarray*}&&
E_-:= \left\{ x=(x',x_n)\in\R^n {\mbox{ s.t. }}
x_n\le -\epsilon - \frac{2}{\pi}\,\arctan \big(\delta \,|x'|^2\big)
\right\}\\ {\mbox{and }}&&
E_+:= \left\{ x=(x',x_n)\in\R^n {\mbox{ s.t. }}
x_n\ge \epsilon + \frac{2}{\pi}\,\arctan \big(\delta \,|x'|^2\big)
\right\}.
\end{eqnarray*}
Notice that the horizontal component of the normal to~$\partial E$
is of size~$O(\delta)$ and the curvatures
are of size~$O(\delta^2)$. Thus, we are in the position of
exploiting Proposition~\ref{STIMA:CUR:FLAT}, from which we
obtain the desired result.
\end{proof}

\section{Neckpinch}\label{neckpinch}

In this section we prove our main result, that is Theorem \ref{thm-neckpinch}. More precisely, we provide an example of surface evolving by fractional mean curvature flow, which develops a singularity before it can shrink to a point. For the classical mean curvature flow, this phenomenon appears in dimensions $n\geq 3$ (that is, for at least $2$-dimensional surfaces) in the neckpinch singularity described in the introduction. In this example, the fact that the dimension of the surface is larger than 2 is crucial to have positive mean curvature in the neck. 

Interestingly, when we consider the fractional mean curvature flow, this neckpinch singularity can be observed also in dimension $n=2$, providing hence a counterexample to the Grayson Theorem \cite{Gr87}, which states that any smooth embedded curve in the plane evolving by (classical) MCF shrinks to a point. The heuristic reason for this is that, thanks to its nonlocality, the fractional mean curvature of a very thin neck, is strictly positive also in dimension $2$: Indeed if we ``sit'' on the boundary of the neck we see much more complement of $E$ than $E$ itself.

To build our example, we start by recalling the following fact that was proved in \cite{SaezV}. 
\begin{lem}[Lemma 2 and Corollary 3 in \cite{SaezV}]\label{ball}
The fractional mean curvature of the ball of radius $R$ is equal, up to dimensional constants, to $R^{-s}$.

More precisely, for any $x\in \partial B_1$
$$H_{B_1}^s(x)=\bar \omega,$$
for some $\bar\omega>0$, and for any $x\in \partial B_R(0)$,
$$H_{B_R}^s(x)=\bar \omega R^{-s}.$$

Moreover, if we set $R(t):=(R_0^{s+1}-(\bar\omega(1+s)) t)^{\frac{1}{s+1}}$, then $B_{R(t)}$ is a solution to the fractional mean curvature flow starting from $B_{R_0}$ and it collapses to a point in the finite time
\begin{equation}\label{Tball}
T_{B_{R_0}}=\frac{R_0^{s+1}}{\bar \omega(s+1)}.
\end{equation}
\end{lem}
Observe that, while for the classical mean curvature flow, the extinction time of a sphere of radius $R_0$ is proportional to $R_0^2$, in the fractional case it is proportional to $R_0^{s+1}$.

We can now give the proof of our main result.

\begin{proof}[Proof of Theorem \ref{neckpinch}]
We consider now the set $E_{\epsilon}$ defined in Corollary \ref{arctan}:

$$ E_\epsilon:= \left\{ x=(x',x_n)\in\R^n {\mbox{ s.t. }}
|x_n|< \epsilon + \frac{2}{\pi}\,\arctan \big(\delta \,|x'|^2\big)
\right\} . $$

We know that there exists $\overline \epsilon$ and $\overline \delta$ positive such that, for any $0<\epsilon\leq \overline \epsilon$ and $0<\delta \leq \overline \delta$

\begin{equation}\label{H_s>0} \inf_{x\in \partial E_\epsilon} H^s_{E_{\epsilon}}(x) \geq c_0> 0, \end{equation}
for some $c_0$ depending only on $n$ and $s$.

Let now $\kappa$ and $\epsilon_0$ be two positive parameters satisfying
\begin{equation}\label{k-epsilon}
\kappa <c_0\quad \mbox{\and}\quad \epsilon_0 <\min\left\{\bar \epsilon, \frac{1}{4}\kappa T_{B_{1}}\right\},
\end{equation}
where $T_{B_{1}}$ is the extinction time of the ball of radius $1$ given in \eqref{Tball}.

The idea is to consider the set $E_{\epsilon_0}$ and to make it evolve with constant velocity $\kappa$ in the inner vertical direction. 
More precisely, we set
$$\epsilon(t):=\epsilon_0-\kappa\,t,$$

and, for any $t$, we consider the set
\begin{equation}\label{E-t}
E_{\epsilon(t)}:=\left\{ x=(x',x_n)\in\R^n {\mbox{ s.t. }}
|x_n|< \epsilon(t) + \frac{2}{\pi}\,\arctan \big(\delta \,|x'|^2\big)
\right\} . 
\end{equation}

Hence, we have that any point $x\in \partial E_{\epsilon(t)}$ satisfies
$$\partial_ t x\cdot \nu =V \cdot \nu,$$
where 
\[V=\begin{cases} -\kappa e_n &\mbox{if}\;\;x_n>0\\
\kappa e_n &\mbox{if}\;\;x_n<0.
\end{cases}
\]

%If we call $F_\epsilon:\partial{E_{\epsilon_0}}\times [0,T)\rightarrow \R^n$ the family of immersion for which $F_\epsilon(\partial E_{\epsilon_0},t)=\partial E_{\epsilon(t)}$, then $F_\epsilon$ satisfies
%%
%\begin{equation*}
%\begin{cases}
%\partial_t F_\epsilon(p,t)\cdot \nu(p,t)=-\kappa e_n\cdot \nu(p,t)&p\in E_{\epsilon_0},\,t\geq 0\\
%F_\epsilon(p,0)=p &p\in E_{\epsilon_0}.
%\end{cases}
%\end{equation*}
%
Thus,%
$$\partial_t x\cdot \nu \geq -\kappa> -c_0\geq -H_{E_{\epsilon(t)}}^s,$$
where in the last inequality we have used \eqref{H_s>0} and the fact that $E_{\epsilon(t)} \subset E_{\epsilon_0}$ for any $t>0$.
Therefore, the set $E_{\epsilon(t)}$ is a smooth supersolution (and hence in particular a viscosity supersolution) to \eqref{NMCF}.

By the definition of the set $E_{\epsilon_0}$ we have that the infimum distance between the two disconnected components of its boundary $\{(x',x_n)\in \R^n \,\,\mbox{s.t.}\,\,x_n=\epsilon_0 + \arctan{(\delta |x'|^2)}\}$ and $\{(x',x_n)\in \R^n \,\,\mbox{s.t.}\,\,x_n=-\epsilon_0 - \arctan{(\delta |x'|^2)}\}$ is attained at the points $(0,\dots,0,\epsilon_0)$ and $(0,\dots,0,-\epsilon_0)$. Since $E_{\epsilon(t)}$ evolves with constant negative velocity $\kappa$ along the  inner vertical direction, we deduce that the singular time for $E_{\epsilon(t)}$ is given by
\begin{equation}\label{T_E}
T_{E_{\epsilon(t)}}=\frac{2\epsilon_0}{\kappa}.
\end{equation}

Let now consider any closed set $A_0$ with the following properties:
\begin{enumerate}
\item $A_0$ is rotationally symmetric around the $x_1$ axis;
\item $A_0$ is symmetric with respect to the $x_1=0$ hyperplane;
\item $A_0$ is contained in $E_{\epsilon_0}$;
\item $A_0$ contains two balls $B_{1}^-$ and $B_{1}^+$ of radius $1$ centered at $(-L,0,\dots,0)$ and $(L,0,\dots,0)$ respectively, where $L$ is chosen large enough so that (1) and (2) are both satisfied.
\end{enumerate}

We consider now the fractional mean curvature flow $A_t$ starting from $A_0$. By uniqueness, $A_t$ retains the symmetries of $A_0$. By the comparison principle (Proposition \ref{comparison}), $A_t$ must be contained in $E_{\epsilon(t)}$. Moreover it must contain the evolutions $B_{1,t}^-$ and $B_{1,t}^+$ of the two balls $B_{1}^-$ and $B_{1}^+$.

On the one hand, since $A_t$ is contained in $E_{\epsilon(t)}$, using \eqref{T_E} and the choice of $\epsilon_0$ \eqref{k-epsilon}, we deduce that at any time $t>T_A$, where
$$
T_{A}= \frac{2\epsilon_0}{\kappa}\leq \frac{1}{2}T_{B_{1}},
$$
the $x_1=0$ cross section of $A_t$ is empty. 

On the other hand, by assumption (2), at the same time, $A_t$ contains two balls with positive radius in the $x_1>0$ and $x_1<0$ half-spaces respectively. This shows that, at some time smaller than $T_A$, the set $A_t$ splits into two symmetric disconnected components, hence it cannot shrink to a point. This concludes the proof of Theorem~\ref{thm-neckpinch}.
\end{proof}

\end{document}